\documentclass[12pt,a4paper]{amsart}

\usepackage{enumerate}
\usepackage{amsmath,amsthm,verbatim,amssymb,amsfonts,amscd,graphicx}
\usepackage{tikz-cd} 
\usepackage{graphics}
\usepackage{hyperref}
\usepackage{mathrsfs}
\usepackage{faktor}
\usepackage[margin=2.9cm]{geometry}
\usepackage[latin2]{inputenc}
\theoremstyle{plain}
\newtheorem{theorem}{Theorem}
\newtheorem*{theorem*}{Theorem}
\newtheorem{corollary}{Corollary}
\newtheorem{lemma}{Lemma}
\newtheorem{proposition}{Proposition}

\theoremstyle{definition}

\begin{document}
\title{Epsilon multiplicity for graded algebras}
\author{Suprajo Das}
\address{Suprajo Das, Department of Mathematics, University of Missouri, Columbia, MO 65211, USA}
\email{sdpw6@mail.missouri.edu}
\begin{abstract}
  The notion of $\varepsilon$-multiplicity was originally defined by Ulrich and Validashti as a limsup and they used it to detect integral dependence of modules. It is important to know if it can be realized as a limit. In this article we show that the relative epsilon multiplicity of reduced standard graded algebras over an excellent local ring exists as a limit. We also obtain some important special cases of Cutkosky's results concerning $\varepsilon$-multiplicity, as corollaries of our main theorem.
\end{abstract}
\maketitle
\section{Introduction}
The purpose of this paper is to prove a very general theorem on the $\varepsilon$-multiplicity of graded algebras over a local ring. The idea of $\varepsilon$-multiplicity originates in the works of Kleiman, Ulrich and Validashti. We now recall their definition of relative $\varepsilon$-multiplicity as introduced in \cite{BJ2} and some related notions.
\par Fix a Noetherian local ring $R$ with maximal ideal $m_R$. We say $T = \bigoplus\limits_{n\geq 0}T_n$ is a standard graded $R$-algebra if $T$ is a graded $R$-algebra with $T_0 = R$ which is generated by finitely many homogeneous elements in $T_1$. Now let $$A = \bigoplus\limits_{n=0}^{\infty}A_n \subset B = \bigoplus\limits_{n=0}^{\infty}B_n$$ be a graded inclusion of standard graded $R$-algebras. Then the \emph{relative $\varepsilon$-multiplicity of $A$ and $B$} is defined to be 
\begin{align*}
    \varepsilon\left(A\mid B\right) &:= \limsup\limits_{n}\dfrac{(\dim B - 1)!}{n^{\dim B - 1}}l_R\left(H^0_{m_R}\left(\faktor{B_n}{A_n}\right)\right)\\
    &= \limsup\limits_{n}\dfrac{(\dim B - 1)!}{n^{\dim B - 1}}l_R\left(\faktor{(A_n\colon_{B_n}m_R^{\infty})}{A_n}\right)
\end{align*}
 The above equality comes from the fact that there are natural $R$-module isomorphisms $$H^0_{m_R}\left(\faktor{B_n}{A_n}\right) \cong \faktor{(A_n\colon_{B_n}m_R^{\infty})}{A_n}$$ for all $n\geq 0$. In \cite{BJ2} it is proven that this invariant is finite under very mild conditions, in particular see Lemma \ref{bound2}. In the same article, the authors use $\varepsilon$-multiplicity to give a `Rees Criteria' for integral dependence of modules or of standard graded algebras. The question of whether $\varepsilon$-multiplicity actually exists as a limit has already been considered by many mathematicians. Some related papers in this direction are \cite{DC1}, \cite{DC3}, \cite{DC4}, \cite{DC5}, \cite{BJ2} and \cite{J}. In \cite{DC1}, Cutkosky has established the following result:
\begin{theorem*}[\cite{DC1}, Theorem $3.2$]
Let $R$ be an analytically unramified local ring of dimension $d$ with maximal ideal $m_R$ and $E$ is a rank $e$ submodule of a finite free $R$-module $F=R^n$. Let $B = R[F]$ be the symmetric algebra of $F$ over $R$, which is isomorphic to the standard graded polynomial ring $$B = R[x_1,\ldots,x_n] = \bigoplus\limits_{k=0}^{\infty} F^k$$ over $R$. We may identify $E$ with a submodule $E^1$ of $B_1$ and let $$A = R[E] = \bigoplus\limits_{k=0}^{\infty}E^k$$ be the graded $R$-subalgebra of $B$ generated by $E^1$ over $R$. Then $$\varepsilon\left(A\mid B\right) = \lim\limits_{k\to\infty}\dfrac{(d+n-1)!}{k^{d+n-1}}l_R\left(\faktor{(E^k\colon_{F^k}m_R^{\infty})}{E^k}\right)$$ exists as a finite limit.
\end{theorem*}
In our paper, we extend this result by allowing $A$ and $B$ to be reduced standard graded algebras over an excellent local ring $R$. We now state the main theorem from our paper:
\begin{theorem*}[Theorem \ref{maintheorem} and Proposition \ref{Hilb}]
Suppose that $R$ is an excellent local ring with maximal ideal $m_R$, $B =\oplus_{n\geq 0}B_n$ is a reduced standard graded $R$-algebra and $A = \oplus_{n\geq 0}A_n$ is a standard graded $R$-subalgebra of $B$. Also assume that one of the following conditions hold:
\begin{itemize}
    \item[(i)] $R$ is a field or
    \item[(ii)] $P\cap R \neq m_R$ for all minimal primes $P$ of $B$.
\end{itemize}
Then $$\varepsilon\left(A\mid B\right) = \lim\limits_{n\to\infty}\dfrac{(\dim B - 1)!}{n^{\dim B - 1}}l_R\left(\faktor{(A_n\colon_{B_n}m_R^{\infty})}{A_n}\right)$$ exists as a finite limit.
\end{theorem*}
It is well known that the Hilbert-Samuel multiplicity is always an integer however this does not hold true for the epsilon multiplicity. A surprising example given in \cite{DC3}, shows that this limit can even be an irrational number. This shows that the sequence associated to the $\varepsilon$-multiplicity cannot have polynomial growth eventually, unlike the classical Hilbert-Samuel function.

\section{Limits for graded algebras over a local domain}    
We adopt all the notations and conventions about convex geometry from section $6.1.$ of \cite{DC1} and quote two results verbatim.
\begin{theorem}[\cite{DC1}, Theorem $6.1.$]\label{Okounkov1}
Suppose that $S$ is strongly nonnegative. Then $$\lim\limits_{n\to \infty}\dfrac{\#S_{m(S)n}}{n^{q(S)}} = \dfrac{\mathrm{vol}_{q(S)}\Delta(S)}{\mathrm{ind}(S)}.$$
\end{theorem}
\begin{theorem}[\cite{DC1}, Theorem $6.2.$]\label{Okounkov2}
Suppose that $p$ is a positive integer such that there exists a sequence $\{n_i\}_{i\in\mathbb{N}}$ of positive integers with $\lim\limits_{i\to \infty} n_i = \infty$ such that the sequence $\left\{\dfrac{\#S_{m(S)n_i}}{n_i^p}\right\}_{i\in\mathbb{N}}$ is bounded. Then $S$ is strongly nonnegative with $q(S)\leq p$.
\end{theorem}

\begin{lemma}\label{lemma1}
Suppose that $R$ is a Noetherian local ring with maximal ideal $m_R$ and let $B = \bigoplus\limits_{k\geq 0}B_k$ be a standard graded $R$-algebra. Let $$m_B = m_RB + \left(\bigoplus\limits_{k\geq 1}B_k\right)$$ denote the unique homogeneous maximal ideal of $B$. Then for any $e\in \mathbb{Z}_{\geq 1}$, we have $$m_B^{k(e+1)}\cap B_k = m_R^{ke}B_k \quad \forall k\in \mathbb{N}.$$
\end{lemma}
\begin{proof}
For any $n\in \mathbb{Z}_{\geq 1}$, we have that
\begin{align*}
    m_B^n &= \left(m_RB + \left(\bigoplus\limits_{k\geq 1}B_k\right)\right)^n\\
    &= \sum\limits_{j=0}^n m_R^{n-j}\left(\bigoplus\limits_{k\geq 1}B_k\right)^{j}\\
    &= \sum\limits_{j=0}^n m_R^{n-j}\left(\bigoplus\limits_{k\geq j}B_k\right)\\
    &= \left(\bigoplus\limits_{k=0}^{n-1} m_R^{n-k}B_k\right)\oplus \left(\bigoplus\limits_{k\geq n} B_k\right).
\end{align*}
Thus we see that 
\[ m_B^n\cap B_k = 
   \begin{cases} 
      m_R^{n-k}B_k &\forall\, 0\leq k\leq n-1 \\
      B_k &\forall\, k\geq n 
   \end{cases}
\]
The lemma now follows by taking $n=k(e+1)$.
\end{proof}

\begin{lemma}\label{lemma2}
Let $T$ be a Noetherian domain and let $\mathfrak{m}$ be a maximal ideal of $T$. Then $$\mathfrak{m}^nT_{\mathfrak{m}} \cap T = \mathfrak{m}^n \quad \forall n\in \mathbb{N}.$$
\end{lemma}
\begin{proof}
Fix an $n\in \mathbb{N}$ and observe that $\mathfrak{m}^n$ is $\mathfrak{m}$-primary since $\sqrt{\mathfrak{m}^n}=\mathfrak{m}$ is a maximal ideal. Now
\begin{align*}
    a \in \mathfrak{m}^nT_{\mathfrak{m}} \cap T &\iff as \in \mathfrak{m}^n \,\text{for some $s\in T\setminus \mathfrak{m}$}\\
    &\iff a\in \mathfrak{m}^n \,\text{or}\,s\in \sqrt{\mathfrak{m}^n}=\mathfrak{m} \quad (\text{$\mathfrak{m}^n$ is $\mathfrak{m}$-primary})\\
    &\iff a\in \mathfrak{m}^n \quad (s\in T\setminus \mathfrak{m})
\end{align*}
\end{proof}
\begin{lemma}\label{comp1}
Let $R$ be a Noetherian local ring with maximal ideal $m_R$, $R[x_1,\ldots,x_n]$ be a polynomial ring over $R$ and $I$ be a homogeneous ideal in $R[x_1,\ldots,x_n]$. Let $\mathfrak{m}$ denote the homogeneous maximal ideal of $R[x_1,\ldots,x_n]$ and let $$B = \dfrac{R[x_1,\ldots,x_n]}{I}.$$ Then the $\mathfrak{m}$-adic completion of $B$ is $$\hat{B}\cong \dfrac{\hat{R}[[x_1,\ldots,x_n]]}{I\hat{R}[[x_1,\ldots,x_n]]},$$ where $\hat{R}$ is the $m_R$-adic completion of $R$ and $\hat{R}[[x_1,\ldots,x_n]]$ is the power series ring over $\hat{R}$ in $x_1,\ldots,x_n$.
\end{lemma}
\begin{proof}
Let $u_1,\ldots,u_d$ be the generators of $m_R$. Then, by \cite[Corollary $5$, Page $171$]{Mat1}, the $\mathfrak{m}$-adic completion of $R[x_1,\ldots,x_n]$ is $$\widehat{R[x_1,\ldots,x_n]}\cong \dfrac{R[x_1,\ldots,x_n][[y_1,\ldots,y_d,z_1,\ldots,z_n]]}{(y_1-u_1,\ldots,y_d-u_d,z_1-x_1,\ldots,z_n-x_n)}$$ where $R[x_1,\ldots,x_n][[y_1,\ldots,y_d,z_1,\ldots,z_n]]$ is the power series ring in $y_1,\ldots,y_d,z_1,\ldots,z_n$ over $R[x_1,\ldots,x_n]$. Rearranging terms, we see that
\begin{align*}
\widehat{R[x_1,\ldots,x_n]} &\cong \dfrac{R[x_1,\ldots,x_n][[y_1,\ldots,y_d,z_1,\ldots,z_n]]}{(y_1-u_1,\ldots,y_d-u_d,z_1-x_1,\ldots,z_n-x_n)}\\
&\cong \dfrac{\dfrac{R[[y_1,\ldots,y_d]]}{(y_1-u_1,\ldots,y_d-u_d)}[x_1,\ldots,x_n][[z_1,\ldots,z_n]]}{(z_1-x_1,\ldots,z_n-x_n)}\\
&\cong \hat{R}[[x_1,\ldots,x_n]].
\end{align*}
Thus
\begin{equation*}
    \hat{B} \cong \dfrac{\widehat{R[x_1,\ldots,x_n]}}{I\widehat{R[x_1,\ldots,x_n]}} \cong \dfrac{\hat{R}[[x_1,\ldots,x_n]]}{I\hat{R}[[x_1,\ldots,x_n]]}.
\end{equation*}
\end{proof}
\begin{lemma}\label{comp2}
Suppose that $R$ is an excellent local ring with maximal ideal $m_R$ and let $$B = \dfrac{R[x_1,\ldots,x_n]}{I}$$ where $I$ is a radical homogeneous ideal in the polynomial ring $R[x_1,\ldots,x_n]$ with $I\cap R = (0)$. Let $\hat{R}$ be the $m_R$-adic completion of $R$ and let $$C = \dfrac{\hat{R}[x_1,\ldots,x_n]}{I\hat{R}[x_1,\ldots,x_n]}.$$ Then the localisation of $C$ at its graded maximal ideal is analytically unramified, the ideal $I\hat{R}[x_1,\ldots,x_n]$ has an irredundant primary decomposition $$I\hat{R}[x_1,\ldots,x_n] = \bigcap\limits_{i=1}^t P_i$$ where the $P_i$ are homogeneous prime ideals in $\hat{R}[x_1,\ldots,x_n]$ and the localisation of $$\dfrac{C}{P_iC}\cong \dfrac{\hat{R}[x_1,\ldots,x_n]}{P_i}$$ at its graded maximal ideal is analytically irreducible for $1\leq i\leq r$.
\end{lemma}
\begin{proof}
Let $\hat{C}$ be the $(m_R + (x_1,\ldots,x_n))$-adic completion of $C$. From lemma \ref{comp1}, we have that $$\hat{C}\cong \dfrac{\hat{R}[[x_1,\ldots,x_n]]}{I\hat{R}[[x_1,\ldots,x_n]]} \cong \hat{B}.$$ $\hat{C}$ is reduced since $B$ is excellent and reduced. Thus $I\hat{R}[[x_1,\ldots,x_n]]$ is a radical ideal in $\hat{R}[[x_1,\ldots,x_n]]$. Now $\hat{R}[[x_1,\ldots,x_n]]$ is the $(m_{\hat{R}}+(x_1,\ldots,x_n))$-adic completion of $\hat{R}[x_1,\ldots,x_n]$. Since the completion of a local ring is faithfully flat, then by \cite[(4.C)(iii), Page $28$]{Mat1}, $$\left(I\hat{R}[[x_1,\ldots,x_n]]\right)\cap \hat{R}[x_1,\ldots,x_n]_{(m_{\hat{R}}+(x_1,\ldots,x_n))} = I\hat{R}[x_1,\ldots,x_n]_{(m_{\hat{R}}+(x_1,\ldots,x_n))}.$$ Since $I\hat{R}[x_1,\ldots,x_n]$ is a homogeneous ideal in $\hat{R}[x_1,\ldots,x_n]$, so $$I\hat{R}[[x_1,\ldots,x_n]]\cap \hat{R}[x_1,\ldots,x_n] = I\hat{R}[x_1,\ldots,x_n].$$ Hence we get an injection of rings $$C = \dfrac{\hat{R}[x_1,\ldots,x_n]}{I\hat{R}[x_1,\ldots,x_n]}\hookrightarrow \dfrac{\hat{R}[[x_1,\ldots,x_n]]}{I\hat{R}[[x_1,\ldots,x_n]]}\cong \hat{C}$$ and thus $C$ is reduced since $\hat{C}$ is reduced. We have that $C_{(m_{\hat{R}}+(x_1,\ldots,x_n))}$ is analytically unramified since the reduced ring $\hat{C}$ is the $(m_{\hat{R}}+(x_1,\ldots,x_n))$-adic completion of $C$.
\par By \cite[Theorem $9$, Page $153$]{Zar}, the radical homogeneous ideal $I\hat{R}[x_1,\ldots,x_n]$ has a primary decomposition $$I\hat{R}[x_1,\ldots,x_n] = \bigcap\limits_{i=1}^t P_i$$ where $P_i$ are the homogeneous prime ideals in $\hat{R}[x_1,\ldots,x_n]$ which are minimal primes of $I\hat{R}[x_1,\ldots,x_n]$. For a particular $i$, $1\leq i\leq t$, let $U = \faktor{\hat{R}}{P_i\cap \hat{R}}$ which is a complete local domain and let $Q = P_i U[x_1,\ldots,x_n]$. We have that $$\dfrac{C}{P_i C} \cong \dfrac{\hat{R}[x_1,\ldots,x_n]}{P_i}\cong \dfrac{U[x_1,\ldots,x_n]}{Q}.$$ Let $V = \faktor{U[x_1,\ldots,x_n]}{Q}$ and we will now show that $\hat{V}$ is a domain, where $\hat{V}$ denotes the $(m_U + (x_1,\ldots,x_n))$-adic completion of $V$. From lemma \ref{comp1}, we have that $$\hat{V}\cong \dfrac{U[[x_1,\ldots,x_n]]}{QU[[x_1,\ldots,x_n]]}.$$ Let $V^*$ be the $(x_1,\ldots,x_n)$-adic completion of $V$. From \cite[Corollary $5$, Page $171$]{Mat1}, it follows that $$V^*\cong \dfrac{U[[x_1,\ldots,x_n]]}{QU[[x_1,\ldots,x_n]]} \cong \hat{V}.$$ Also by \cite[Proposition $10.22.(ii)$, Page $111$]{Atiyah}, we get that $$\mathrm{gr}_{(x_1,\ldots,x_n)}(V^*)\cong \mathrm{gr}_{(x_1,\ldots,x_n)}(V).$$ Since $V$ is graded, we have that $$\mathrm{gr}_{(x_1,\ldots,x_n)}(V)\cong V$$ and the latter is a domain. Further since $V$ is graded, it also gives us that $$\bigcap\limits_{s\geq 0}(x_1,\ldots,x_n)^sV = 0.$$ Thus the local ring $\hat{V}\cong V^*$ is a domain by \cite[Theorem $1$, Page $249$]{Zar}
\end{proof}
The proof of Theorem $\ref{main}$ uses methods of the proof of Theorem $6.3.$ in \cite{DC1}. Here we must blow up a different ideal and construct a different valuation to make the argument work in our case. Also, we must make a more delicate analysis of the grading in our argument.
\begin{theorem}\label{main}
Suppose that $R$ is a Noetherian complete local domain with maximal ideal $m_R$, $B = \bigoplus\limits_{n\geq 0}B_n$ is a standard graded $R$-algebra which is also a domain and let $A = \bigoplus\limits_{n\geq 0}A_n$ be a graded $R$-subalgebra of $B$. Suppose that $A_1\neq 0$ and that $p\in \mathbb{Z}_{\geq 1}$ is such that for all $c\in \mathbb{Z}_{\geq 1}$, there exists $\gamma_c\in \mathbb{R}_{>0}$ such that 
\begin{equation}\label{limsup}
    l_R\left(\faktor{A_n}{(m_R^{cn}B)\cap A_n}\right)<\gamma_cn^p
\end{equation}
for all $n\geq 0$. Then for any fixed positive integer $c$, $$\lim_{n\to \infty}\dfrac{l_R\left(\faktor{A_n}{(m_R^{cn}B)\cap A_n}\right)}{n^p}$$ exists.
\end{theorem}
We now prove Theorem $\ref{main}$. Let $c>0$ be a fixed positive integer. Let $m_B$ denote the homogeneous maximal ideal of $B$ and it follows from Lemma \ref{comp2} that $B_{m_B}$ is analytically irreducible. Let $$\pi\colon X \to \text{Spec}(B_{m_B})$$ be the normalization of the blow up of the unique closed point $\{m_BB_{m_B}\}$ of $\text{Spec}(B_{m_B})$. The morphism $\pi$ is of finite type since $R$ is excellent. The fibre $\pi^{-1}\left(\{m_BB_{m_B}\}\right)$ is a closed subscheme of codimension $1$ in $X$. Since $X$ is normal, its singular locus has codimension $\geq 2$. So there exists a closed point $x\in \pi^{-1}\left(\{m_BB_{m_B}\}\right)$ such that $\mathcal{O}_{X,x}$ is a regular local ring and we set $(S,m_S)$ to be this local ring. From the construction, we see that $S$ dominates $B_{m_B}$, $\dim S = \dim B$, $S$ is essentially of finite type over $B$ with $\text{QF}(B)=\text{QF}(S)$ and the residue field map $$\faktor{B_{m_B}}{m_BB_{m_B}} = \faktor{R}{m_R}\hookrightarrow \faktor{S}{m_S}$$ is a finite field extension. Let $l = \left[S/m_S\colon R/m_R\right]<\infty.$
\par Let $d = \dim B = \dim S$ and let $y_1,\ldots,y_d$ be a regular system of parameters in $S$. Fix $\mathbb{Q}$-linearly independent real numbers $\lambda_1,\ldots,\lambda_d$ such that $\lambda_i\geq 1$ for all $1\leq i\leq d$. Let $\mathcal{C}(S)$ be a coefficient set of $S$. Since $S$ is a regular local ring, for any $r\in \mathbb{Z}_{\geq 1}$ and $f\in S\setminus\{0\}$, there is a unique expression $$f = \sum\limits_{(n_1,\ldots,n_d)\in \mathbb{N}^d}a_{n_1,\ldots,n_d}y_1^{n_1}\cdots y_d^{n_d} + f_r$$ where $a_{n_1,\ldots,n_d}\in \mathcal{C}(S)$, $f_r\in m_S^r$ and $n_1+\cdots +n_d<r$ for all $(n_1,\ldots,n_d)$ appearing in the sum. We can take $r$ large enough so that $n_1\lambda_1+\cdots +n_d\lambda_d<r$ for some term with $a_{n_1,\ldots,n_d}\neq 0$. We define a valuation $\nu$ of $\text{QF}(B)$ which dominates $B_{m_B}$ with value group $\Gamma_{\nu} = \mathbb{Z}\lambda_1+\cdots + \mathbb{Z}\lambda_d \subset \mathbb{R}$, by prescribing $$\nu(f)=\min\{n_1\lambda_1+\cdots +n_d\lambda_d\mid a_{n_1,\cdots,n_d}\neq 0\}.$$ Since there is a unique monomial giving the minimum, we have that $$\faktor{S}{m_S}=\faktor{V_{\nu}}{m_{\nu}}.$$ For $\lambda\in\mathbb{R}_{\geq 0}$, define valuation ideals in the valuation ring $V_{\nu}$ of $\nu$, by
\begin{align*}
    K_{\lambda} &= \{f\in\mathrm{QF}(B)\mid \nu(f)\geq \lambda\},\\
    K_{\lambda}^+ &= \{f\in\mathrm{QF}(B)\mid \nu(f) > \lambda\}.
\end{align*}
There exists a constant $\alpha\in\mathbb{Z}_{\geq 1}$ such that 
\begin{equation}\label{eq0}
    K_{\alpha n} \cap B_{m_B} \subset m_B^nB_{m_B} \quad \forall n\in \mathbb{N}.
\end{equation}
The proof of this formula follows from \cite[Lemma $4.3.$]{DC2} and the fact that $B_{m_B}$ is analytically irreducible is necessary for the validity of this formula. Set $\beta = \alpha (c+1)$ and for all $n\in \mathbb{N}$, we have that 
\begin{align*}
    K_{\beta n}\cap A_n &= \left(\left(\left(K_{\alpha (c+1) n}\cap B_{m_B}\right)\cap B\right)\cap B_n\right)\cap A_n &&\\
    &\subset \left(\left(m_B^{(c+1)n}B_{m_B}\cap B\right)\cap B_n\right)\cap A_n &&\quad(\text{from (\ref{eq0})})\\
    &= \left(m_B^{(c+1)n}\cap B_n\right)\cap A_n &&\quad(\text{from Lemma \ref{lemma2}})\\
    &= \left(m_R^{cn}B_n\right)\cap A_n &&\quad(\text{from Lemma \ref{lemma1}})\\
    &= \left(m_R^{cn}B\right)\cap A_n &&
\end{align*}
Let $\overline{A}_n = \left(m_R^{cn}B\right)\cap A_n$. We have that
\begin{equation}\label{eq1}
    l_R\left(\faktor{A_n}{\left(m_R^{cn}B\right)\cap A_n}\right) = l_R\left(\faktor{A_n}{K_{\beta n}\cap A_n}\right) - l_R\left(\faktor{\overline{A}_n}{K_{\beta n}\cap \overline{A}_n}\right)
\end{equation}
For $t\geq 1$, define
\begin{align*}
    \Gamma^{(t)} &= \left\{(n_1,\ldots,n_d,n)\in\mathbb{N}^{d+1}\mid \dim_{R/m_R}\left(\dfrac{A_n\cap K_{n_1\lambda_1+\cdots+n_d\lambda_d}}{A_n\cap K^+_{n_1\lambda_1+\cdots+n_d\lambda_d}}\right)\geq t\;\mathrm{and}\; n_1+\cdots+n_d\leq \beta n\right\},\\
    \overline{\Gamma}^{(t)} &= \left\{(n_1,\ldots,n_d,n)\in\mathbb{N}^{d+1}\mid \dim_{R/m_R}\left(\dfrac{\overline{A}_n\cap K_{n_1\lambda_1+\cdots+n_d\lambda_d}}{\overline{A}_n\cap K^+_{n_1\lambda_1+\cdots+n_d\lambda_d}}\right)\geq t\;\mathrm{and}\; n_1+\cdots+n_d\leq \beta n\right\}.
\end{align*}
For all $n\in \mathbb{N}$ and $\lambda\in\mathbb{R}_{\geq 0}$, we have natural $R/m_R$-vector space inclusions
\begin{align*}
    \faktor{A_n\cap K_{\lambda}}{A_n\cap K^+_{\lambda}}&\hookrightarrow \faktor{K_{\lambda}}{K^+_{\lambda}}\cong \faktor{V_{\nu}}{m_{\nu}},\\
    \faktor{\overline{A}_n\cap K_{\lambda}}{\overline{A}_n\cap K^+_{\lambda}}&\hookrightarrow \faktor{K_{\lambda}}{K^+_{\lambda}}\cong \faktor{V_{\nu}}{m_{\nu}}.
\end{align*}
Therefore $\Gamma^{(t)}=\emptyset$ and $\overline{\Gamma}^{(t)}=\emptyset$ for all $t>l$. We also have that if $n_1+\cdots+n_d\leq \beta n$ then
\begin{align*}
    \dim_{R/m_R}\left(\dfrac{A_n\cap K_{n_1\lambda_1+\cdots+n_d\lambda_d}}{A_n\cap K^+_{n_1\lambda_1+\cdots+n_d\lambda_d}}\right) &= \#\left\{t\mid (n_1,\ldots,n_d,n)\in \Gamma^{(t)}\right\},\\
    \dim_{R/m_R}\left(\dfrac{\overline{A}_n\cap K_{n_1\lambda_1+\cdots+n_d\lambda_d}}{\overline{A}_n\cap K^+_{n_1\lambda_1+\cdots+n_d\lambda_d}}\right) &= \#\left\{t\mid (n_1,\ldots,n_d,n)\in \overline{\Gamma}^{(t)}\right\}.
\end{align*}
Further note that $$n_1\lambda_1+\cdots+n_d\lambda_d<\beta n \implies n_1+\cdots+n_d\leq \beta n$$ since $\lambda_i\geq 1$ for all $1\leq i\leq d$. Thus
\begin{align}
    l_R\left(\faktor{A_n}{K_{\beta n}\cap A_n}\right) &= \sum\limits_{\substack{(n_1,\ldots,n_d)\in \mathbb{N}^d \\ n_1\lambda_1+\cdots+n_d\lambda_d < \beta n}}\dim_{R/m_R}\left(\dfrac{A_n\cap K_{n_1\lambda_1+\cdots+n_d\lambda_d}}{A_n\cap K^+_{n_1\lambda_1+\cdots+n_d\lambda_d}}\right) = \sum\limits_{t=1}^l \#\Gamma_n^{(t)},\label{eq2}\\
    l_R\left(\faktor{\overline{A}_n}{K_{\beta n}\cap \overline{A}_n}\right) &= \sum\limits_{\substack{(n_1,\ldots,n_d)\in \mathbb{N}^d \\ n_1\lambda_1+\cdots+n_d\lambda_d < \beta n}}\dim_{R/m_R}\left(\dfrac{\overline{A}_n\cap K_{n_1\lambda_1+\cdots+n_d\lambda_d}}{\overline{A}_n\cap K^+_{n_1\lambda_1+\cdots+n_d\lambda_d}}\right) = \sum\limits_{t=1}^l \#\overline{\Gamma}_n^{(t)}.\label{eq3}
\end{align}
The proof of Lemma $\ref{semigroup}$ is similar to the proof of Lemma $4.4.$ in \cite{DC6}.
\begin{lemma}\label{semigroup}
Suppose that $t\geq 1, 0\neq f\in A_i, 0\neq g\in A_j$ and $$\dim _{R/m_R}\left(\dfrac{A_i\cap K_{\nu(f)}}{A_i\cap K^+_{\nu(f)}}\right)\geq t.$$ Then $$\dim _{R/m_R}\left(\dfrac{A_{i+j}\cap K_{\nu(fg)}}{A_{i+j}\cap K^+_{\nu(fg)}}\right)\geq t.$$ Suppose that $t\geq 1, 0\neq f\in \overline{A}_i, 0\neq g\in \overline{A}_j$ and $$\dim _{R/m_R}\left(\dfrac{\overline{A}_i\cap K_{\nu(f)}}{\overline{A}_i\cap K^+_{\nu(f)}}\right)\geq t.$$ Then $$\dim _{R/m_R}\left(\dfrac{\overline{A}_{i+j}\cap K_{\nu(fg)}}{\overline{A}_{i+j}\cap K^+_{\nu(fg)}}\right)\geq t.$$
\end{lemma}\begin{proof}
We shall only prove the first assertion since the proof of the second assertion is similar to the first one. There exist $f_1,\ldots,f_t \in A_i\cap K_{\nu(f)}$ such that their classes are $\faktor{R}{m_R}$-linearly independent in $\faktor{A_i\cap K_{\nu(f)}}{A_i\cap K^+_{\nu(f)}}.$ We shall show that the classes of $gf_1,\ldots,gf_t$ in $\faktor{A_{i+j}\cap K_{\nu(fg)}}{A_{i+j}\cap K^+_{\nu(fg)}}$ are linearly independent over $\faktor{R}{m_R}$. Suppose that there exist $a_1,\ldots,a_t\in \faktor{R}{m_R}$ such that the class of $a_1gf_1+\cdots+a_tgf_t$ in $\faktor{A_{i+j}\cap K_{\nu(fg)}}{A_{i+j}\cap K^+_{\nu(fg)}}$ is zero. Then
\begin{align*}
    \nu\left(a_1gf_1+\cdots+a_tgf_t\right)&>\nu(fg)\\
    \implies \nu\left(a_1f_1+\cdots+a_tf_t\right)&>\nu(f).
\end{align*}
Therefore the class of $a_1f_1+\cdots+a_tf_t$ in $\faktor{A_i\cap K_{\nu(f)}}{A_i\cap K^+_{\nu(f)}}$ is zero and linear independence gives us that $a_1=\cdots=a_t=0.$
\end{proof}
The proof of Proposition $\ref{semiprop}$ is similar to the proof of Proposition $6.6.$ in \cite{DC1}.
\begin{proposition}\label{semiprop}
Suppose that $t\geq 1, \Gamma^{(t)}\not\subset \{0\}$ and $\overline{\Gamma}^{(t)}\not\subset \{0\}$. Then
\begin{itemize}
    \item[(i)] $\Gamma^{(t)}$ is a subsemigroup of $\mathbb{N}^{d+1}$.
    \item[(ii)] $m\left(\Gamma^{(t)}\right)=1.$
    \item[(iii)] $\Gamma^{(t)}$ is strongly nonnegative with $q\left(\Gamma^{(t)}\right)\leq p.$
    \item[(iv)] $\overline{\Gamma}^{(t)}$ is a subsemigroup of $\mathbb{N}^{d+1}$.
    \item[(v)] $m\left(\overline{\Gamma}^{(t)}\right)=1.$
    \item[(vi)] $\overline{\Gamma}^{(t)}$ is strongly nonnegative with $q\left(\overline{\Gamma}^{(t)}\right)\leq p.$
\end{itemize}
\end{proposition}
\begin{proof}
We shall only indicate the proof of the statements for $\Gamma^{(t)}$ since the proof of the statements for $\overline{\Gamma}^{(t)}$ are the same. The proof of $(i)$ follows from Lemma $\ref{semigroup}$.
\par By assumption $\Gamma_i^{(t)}\neq\emptyset$ for some $i\geq 1$. Thus there exists $0\neq f\in A_i$ such that $$\nu(f) = n_1\lambda_1+\cdots+n_d\lambda_d$$ with $n_1+\cdots+n_d\leq \beta i$ and $$\dim _{R/m_R}\left(\dfrac{A_i\cap K_{\nu(f)}}{A_i\cap K^+_{\nu(f)}}\right)\geq t.$$ By assumption, there exists $0\neq g\in A_1$. Let $$\nu(g) = m_1\lambda_1+\cdots+m_d\lambda_d.$$ After increasing $\beta$ if necessary, we may assume that $m_1+\cdots+m_d\leq \beta$. Thus $$\nu(fg) = \nu(f) + \nu(g) = (m_1+n_1)\lambda_1+\cdots+(m_d+n_d)\lambda_d$$ with $(m_1+n_1)+\cdots+(m_d+n_d)\leq \beta(i+1)$. So $\Gamma_{i+1}^{(t)}\neq \emptyset$ by Lemma $\ref{semigroup}$ and thereby proving $(ii)$.
\par Note that $m_B^n \subset K_{n}\cap B$ for all $n\in \mathbb{N}$ since $\lambda_i\geq 1$ for $1\leq i\leq d$ and thus $$\left(m_R^{\beta n}B\right)\cap A_n \subset m_B^{\beta n} \cap A_n \subset K_{\beta n}\cap A_n.$$ Now for all $n\in \mathbb{N}$, we have
\begin{align*}
\#\Gamma_n^{(t)} &\leq l_R\left(\faktor{A_n}{K_{\beta n}\cap A_n}\right)  &&(\text{from $(\ref{eq2})$})\\
&\leq l_R\left(\faktor{A_n}{\left(m_R^{\beta n}B\right)\cap A_n}\right)\\
&\leq \gamma_{\beta}n^p &&(\text{from $(\ref{limsup})$}).
\end{align*}
The conclusions of $(iii)$ now follow from Theorem \ref{Okounkov2}.
\end{proof}
It thus follows from Theorem \ref{Okounkov1} that the limits $$\lim\limits_{n\to \infty}\dfrac{\#\Gamma_n^{(t)}}{n^p}\;\mathrm{and}\;\lim\limits_{n\to \infty}\dfrac{\#\overline{\Gamma}_n^{(t)}}{n^p}$$ exist. From the previous computations, we also get that
\begin{align*}
    \dfrac{l_R\left(\faktor{A_n}{(m_R^{cn}B)\cap A_n}\right)}{n^p} &= \dfrac{l_R\left(\faktor{A_n}{K_{\beta n}\cap A_n}\right)}{n^p} - \dfrac{l_R\left(\faktor{\overline{A}_n}{K_{\beta n}\cap \overline{A}_n}\right)}{n^p} \quad (\text{by $(\ref{eq1})$}) \\
    &= \sum\limits_{t=1}^l \dfrac{\#\Gamma_n^{(t)}}{n^p} - \sum\limits_{t=1}^l \dfrac{\#\overline{\Gamma}_n^{(t)}}{n^p} \quad (\text{by $(\ref{eq2})$ and $(\ref{eq3})$}).
\end{align*}
We now get the conclusions of Theorem \ref{main} by taking $n\to\infty$.
\begin{lemma}\label{artificial}
Suppose that $R$ is a local ring, $B = \bigoplus\limits_{n\geq 0}B_n$ is a graded $R$-algebra with $B_0 = R$ and $A = \bigoplus\limits_{n\geq 0}A_n$ is a graded $R$-subalgebra of $B$. Suppose that $P$ is a homogeneous prime ideal of $B$. Then for any $n\geq 1$, $$\faktor{A_n}{P\cap A_n} = 0 \iff \left(\faktor{A_n}{P\cap A_n}\right)_{P\cap R} = 0.$$
\end{lemma}
\begin{proof}
Localisation of a zero module is zero which proves one direction. Now suppose that $$\left(\faktor{A_n}{P\cap A_n}\right)_{P\cap R}=0$$ and there exists a non-zero element $x\in \faktor{A_n}{P\cap A_n}.$ Then there exists an element $f\in R\setminus (P\cap R)$ such that $fx = 0$ in $\faktor{A_n}{P\cap A_n}.$ Note that $\faktor{A_n}{P\cap A_n} \subset \faktor{A}{P\cap A}$ which is a domain. Now $f\in R\setminus (P\cap R) \implies f\notin P.$ So the class $\Bar{f}$ of $f$ in $\faktor{R}{P\cap R}\subset \faktor{B}{P}$ is nonzero and therefore $x=0$ which is a contradiction.
\end{proof}
The later part of the proof of Theorem $\ref{main2}$, from $(\ref{similar})$ onwards is similar to the proof of Theorem $6.7.$ in \cite{DC1}. 
\begin{theorem}\label{main2}
Suppose that $R$ is an excellent local ring with maximal ideal $m_R$, $B =\bigoplus\limits_{n\geq 0}B_n$ is a reduced standard graded $R$-algebra and $A = \bigoplus\limits_{n\geq 0}A_n$ is a graded $R$-subalgebra of $B$. Suppose that if $P$ is a minimal prime ideal of $B$ (which is necessarily homogeneous) then $P\cap R \neq m_R$ and if $\faktor{A_1}{P\cap A_1} = 0$ then $\faktor{A_n}{P\cap A_n} = 0$ for all $n\geq 1$. Further suppose that there exists $p\in\mathbb{Z}_{\geq 1}$ such that for all $c\in\mathbb{Z}_{\geq 1}$, there exists $\gamma_c\in\mathbb{R}_{>0}$ such that $$l_R\left(\faktor{A_n}{(m_R^{cn}B)\cap A_n}\right)<\gamma_c n^p$$ for all $n\geq 0$. Then for any fixed integer $c$, $$\lim_{n\to\infty}\dfrac{l_R\left(\faktor{A_n}{(m_R^{cn}B)\cap A_n}\right)}{n^p}$$ exists.
\end{theorem}
\begin{proof}
Let $c>0$ be a fixed positive integer. We first reduce to the case $R$ is complete. Let $\hat{R}$ denote the $m_R$-adic completion of $R$. Since $B$ is a reduced standard graded $R$-algebra, we can write $$B = \dfrac{R[x_1,\ldots,x_n]}{I} = \bigoplus\limits_{n\geq 0}B_n$$ where $R[x_1,\ldots,x_n]$ is a polynomial ring over $R$ and $I$ is a radical homogeneous ideal in $R[x_1,\ldots,x_n]$ such that $I\cap R = (0)$. Now $$B\otimes_R \hat{R} \cong \dfrac{\hat{R}[x_1,\ldots,x_n]}{I\hat{R}[x_1,\ldots,x_n]} \cong \bigoplus\limits_{n\geq 0}\left(B_n\otimes_{R}\hat{R}\right).$$ Lemma \ref{comp2} gives us that $B\otimes_R\hat{R}$ is a standard graded reduced $\hat{R}$-algebra and the localisation of $B\otimes_R\hat{R}$ at its graded maximal ideal is analytically unramified. Also note that $$ A\otimes_R \hat{R}\cong \bigoplus\limits_{n\geq 0}\hat{A}_n$$ is a graded $\hat{R}$-subalgebra of $B\otimes_R\hat{R}$. Moreover we have the following isomorphisms of modules $$\faktor{(A_n\otimes_R\hat{R})}{(m_{\hat{R}}^{cn}(B\otimes_R\hat{R}))\cap (A_n\otimes_R\hat{R})} \cong \left(\faktor{A_n}{(m_R^{cn}B)\cap A_n}\right)\otimes_R \hat{R} \cong \faktor{A_n}{(m_R^{cn}B)\cap A_n}$$ which shows that the relevant lengths do not change after we pass to completion.
\par Let $Q$ be a minimal prime ideal of $B\otimes_R\hat{R}$. We get a commutative diagram as shown below
\[
  \begin{tikzcd}
     R \arrow{r} \arrow{d} & \hat{R} \arrow{d}\\
     A \arrow{r} \arrow{d} & A\otimes_R\hat{R} \arrow{d}\\
     B \arrow{r} & B\otimes_R\hat{R}
  \end{tikzcd}
\]
where all the arrows are natural maps which are also injective. The horizontal maps in the above diagram are also faithfully flat. This is because the natural map $R \to \hat{R}$ is faithfully flat and base change preserves faithful flatness. In particular, the natural map $B \to B\otimes_R \hat{R}$ is flat and flat maps satisfy the going down property. This shows that $Q$ contracts to a minimal prime ideal of $B$. If $$\faktor{(A_1\otimes_R\hat{R})}{Q\cap (A_1\otimes_R\hat{R})} = 0,$$ then the injection $$\faktor{A_1}{(Q\cap B)\cap A_1} \hookrightarrow \faktor{(A_1\otimes_R\hat{R})}{Q\cap (A_1\otimes_R\hat{R})}$$ implies that $$\faktor{A_1}{(Q\cap B)\cap A_1} = 0.$$ Since $Q\cap B$ is a minimal prime ideal of $B$, the assumptions of the theorem further imply that $$\faktor{A_n}{(Q\cap B)\cap A_n}= 0$$ for all $n\geq 1$. Since $R \to \hat{R}$ is faithfully flat, it gives us that $$\faktor{(A_n \otimes_R \hat{R})}{Q \cap (A_n\otimes_R \hat{R})} \cong \left(\faktor{A_n}{(Q\cap B)\cap A_n}\right)\otimes_R \hat{R} = 0$$ for all $n\geq 1$. Again from the commutativity of the diagram and the assumptions of the theorem, we get $$(Q\cap \hat{R})\cap R = (Q\cap B)\cap R \neq m_R.$$ Therefore we must have that $Q\cap \hat{R} \neq m_{\hat{R}}.$
\par Replacing $R$ with $\hat{R}$, $B$ with $B\otimes_R \hat{R}$ and $A$ with $A\otimes_R \hat{R}$, we can thus assume that $R$ is complete and the localisation of $B$ at its graded maximal ideal is analytically unramified. By lemma \ref{comp2}, there exists a minimal primary decomposition $$(0) = \bigcap\limits_{i=1}^t P_i$$ where $P_i$ are the minimal primes of $B$ (which are necessarily homogeneous). For every $1\leq i\leq t$, let $C^i = \faktor{B}{P_i}.$ Observe that $\faktor{R}{P_i \cap R}$ is a complete local domain, $C^i$ is a standard graded $\faktor{R}{P_i \cap R}$-algebra which is also a domain and lemma \ref{comp2} implies that the localization of $C^i$ at its graded maximal ideal is analytically irreducible. Write $C^i = \bigoplus\limits_{n\geq0}C^i_n$ to denote its graded components. Let
\begin{equation}\label{similar}
    \varphi\colon B \hookrightarrow \bigoplus\limits_{i=1}^t \faktor{B}{P_i} = \bigoplus\limits_{i=1}^t C^i =: C
\end{equation} be the natural homomorphism. $\varphi$ is injective since its kernel is $\bigcap\limits_{i=1}^t P_i = 0.$ By the Artin-Rees lemma, there exists a positive integer $k$ such that
\begin{equation*}
\omega_n := \varphi^{-1}\left(m_R^nC\right) = B\cap m_R^nC \subset m_R^{n-k}B
\end{equation*}
for all $n\geq k$. Thus
\begin{equation*}
m_R^nB \subset \omega_n\subset m_R^{n-k}B
\end{equation*}
for all $n\geq k$. We have that $$\omega_n = \varphi^{-1}\left(\bigoplus\limits_{i=1}^t m_R^n C^i\right) = \bigcap\limits_{i=1}^t \left(m_R^n + P_i\right)B.$$ Let $\beta = (k+1)c.$ We have that $$\omega_{\beta n}\subset m_R^{c(k+1)n - k}B\subset m_R^{cn}B$$ for all $n\geq 1$. Thus
\begin{equation}\label{alpha_1}
l_R\left(\faktor{A_n}{(m_R^{cn}B)\cap A_n}\right) = l_R\left(\faktor{A_n}{\omega_{\beta n}\cap A_n}\right) - l_R\left(\faktor{(m_R^{cn}B)\cap A_n}{\omega_{\beta n}\cap (m_R^{cn}B)\cap A_n}\right)
\end{equation}
for all $n\geq 1$. Define the $R$-modules as follows:
\[
  L^j_n = \left.
  \begin{cases}
    R,  &\text{for } 0\leq j\leq t, \; n=0 \\
    A_n, &\text{for } j =0,\; n\geq 1\\
    \left(\bigcap\limits_{i=1}^j \left(m_R^{\beta n} + P_i\right)B\right)\cap A_n, &\text{for } 1\leq j\leq t, \; n\geq 1
  \end{cases}
\right.
\]
Then $$L^j := \bigoplus\limits_{n\geq 0}L^j_n$$ is a graded $R$-subalgebra of $B$. For $0\leq j\leq t-1$ and $n\geq 1$, we have isomorphisms of $R$-modules
\begin{equation*}
\faktor{L^j_n}{L^{j+1}_n} \cong \faktor{L^j_n C^{j+1}_0}{(L^j_n C^{j+1}_0)\cap m_R^{\beta n}C^{j+1}_n}
\end{equation*}
where $L^j_n C^{j+1}_0$ is the image of $L^j_n$ in $C^{j+1}_n$
and $$L^t_n \cong \omega_{\beta n}\cap A_n.$$ Thus
\begin{align}
l_R\left(\faktor{A_n}{\omega_{\beta n}\cap A_n}\right) &= \sum\limits_{j=0}^{t-1}l_R\left(\faktor{L^j_n}{L^{j+1}_n}\right)\nonumber\\
&= \sum\limits_{j=0}^{t-1}l_{\left(\faktor{R}{P_{j+1}\cap R}\right)}\left(\faktor{L^j_n C^{j+1}_0}{(L^j_n C^{j+1}_0)\cap m_R^{\beta n}C^{j+1}_n}\right).\label{alpha_2}
\end{align}
For some fixed $j$ with $0\leq j\leq t-1$, let $$\Bar{R} = \dfrac{R}{P_{j+1}\cap R},\quad \Bar{C} = C^{j+1} = \bigoplus\limits_{n\geq 0}C^{j+1}_n,\quad \Bar{A}_n = L^j_n C^{j+1}_0, \quad \Bar{A} = \bigoplus\limits_{n\geq 0}L^j_n C^{j+1}_0 = \bigoplus\limits_{n\geq 0}\Bar{A}_n.$$ We remind ourselves that $\Bar{R}$ is a complete local domain, $\Bar{C}$ is a standard graded $\Bar{R}$-algebra which is also a domain and the localisation of $\Bar{C}$ at its graded maximal ideal is analytically irreducible. Moreover $\Bar{A}$ is a graded $\Bar{R}$-subalgebra of $\Bar{C}$. Also note that $$l_{\Bar{R}}\left(\faktor{\Bar{A}_n}{m_{\Bar{R}}^{\beta n}\Bar{C}_n\cap \Bar{A}_n}\right)\leq l_R\left(\faktor{A_n}{\omega_{\beta n}\cap A_n}\right)\leq l_R\left(\faktor{A_n}{m_R^{\beta n}B\cap A_n}\right)<\gamma_{\beta}n^p.$$ We shall now establish that if $\Bar{A}_1 = 0$ then $\Bar{A}_n = 0$ for all $n\geq 1$. If $j=0$, then our claim follows from the hypothesis. Suppose that $\Bar{A}_1 = 0$ and $1\leq j\leq t-1$. Then $$m_R^{\beta}A_1 \subset P_{j+1} \implies m_R^{\beta n}A_n \subset P_{j+1} \implies (A_n)_{P_{j+1}\cap R} \subset (P_{j+1}\cap A_n)_{P_{j+1}\cap R}$$ as $P_{j+1}\cap R \neq m_R.$ This implies that $\left(\faktor{A_n}{P_{j+1}\cap A_n}\right)_{P_{j+1}\cap R} = 0$, so $\faktor{A_n}{P_{j+1}\cap A_n} = 0$ by Lemma \ref{artificial}, which in turn shows that $\Bar{A}_n = 0$. So we may assume that $\Bar{A}_1 \neq 0$ and by Theorem \ref{main}, $$\lim\limits_{n\to\infty}\dfrac{l_{\Bar{R}}\left(\faktor{\Bar{A}_n}{m_{\Bar{R}}^{\beta n}\Bar{C}_n\cap \Bar{A}_n}\right)}{n^p}$$ exists. Since this limit exists for all $0\leq j\leq t-1$, $$\lim\limits_{n\to\infty}\dfrac{l_R\left(\faktor{A_n}{\omega_{\beta n}\cap A_n}\right)}{n^p}$$ exists by $(\ref{alpha_2})$. The same argument applied to $(m_R^{cn}B)\cap A_n$ (instead of $A_n$) implies that $$\lim\limits_{n\to\infty}\dfrac{l_R\left(\faktor{(m_R^{cn}B)\cap A_n}{\omega_{\beta n}\cap (m_R^{cn}B)\cap A_n}\right)}{n^p}$$ exists, so $$\lim\limits_{n\to\infty}\dfrac{l_R\left(\faktor{A_n}{m_R^{\beta n}B\cap A_n}\right)}{n^p}$$ exists by $(\ref{alpha_1})$.
\end{proof}
The proof of Lemma $\ref{bound1}$ is similar to that of Lemma $6.9.$ in \cite{DC1}.
 \begin{lemma}\label{bound1}
 Suppose that $R$ is a Noetherian local ring with maximal ideal $m_R$, $B =\bigoplus\limits_{n\geq 0}B_n$ is a standard graded $R$-algebra and $A = \bigoplus\limits_{n\geq 0}A_n$ is a standard graded $R$-subalgebra of $B$. Fix a positive integer $c$. Then there exists a constant $\gamma_c\in \mathbb{R}_{>0}$ such that $$l_R\left(\faktor{A_n}{(m_R^{cn}B)\cap A_n}\right)<\gamma_c n^{\dim A -1}$$ for all $n\geq 1$.
 \end{lemma}
 \begin{proof}
 Let $$H = \bigoplus\limits_{i,j\geq 0}\left(\faktor{m_R^i A_j}{m_R^{i+1} A_j}\right).$$ Then $H$ is a bigraded algebra over the field $\kappa := \faktor{R}{m_R}$. Let $a_1,\ldots,a_u$ be the generators of $m_R$ as an $R$-module and let $b_1,\ldots,b_v$ be the generators of $A_1$ as an $R$-module. Let $$S := \kappa [x_1,\ldots,x_u;y_1,\ldots,y_v]$$ be a polynomial ring and $S$ is bigraded by $\deg (x_i) = (1,0)$ and $\deg (y_j) = (0,1)$. The surjective $\kappa$-algebra homomorphism $S \to H$ defined by $$x_i\to [a_i]\in \faktor{m_R}{m_R^2},\quad y_j \to [b_j] \in \faktor{A_1}{m_R A_1}$$ is bigraded, realizing $H$ as a bigraded $S$-module. Moreover $H \cong \mathrm{gr}_{(m_RA)}(A)$, so that $$\dim_S H = \dim H = \dim \left(\mathrm{gr}_{(m_RA)}(A)\right) \leq \dim A.$$ It can now be deduced from \cite[Theorem $2.4$]{JO} that there exists a positive integer $n_0$ such that for all $n\geq n_0$ implies $$l_R\left(\faktor{A_n}{m_R^{cn}A_n}\right) = \sum\limits_{i=0}^{cn-1} \dim_{\kappa}\left(\faktor{m_R^i A_n}{m_R^{i+1} A_n}\right)$$ is a polynomial in $n$ of degree at most $\dim A -1$, from which the conclusions of the lemma follows.
 \end{proof}
 The following lemma follows as in \cite{BJ2}. For the reader's convenience we give a proof.
 \begin{lemma}\label{bound2}
 Suppose that $R$ is a Noetherian local ring with maximal ideal $m_R$, $B =\bigoplus\limits_{n\geq 0}B_n$ is a standard graded $R$-algebra and $A = \bigoplus\limits_{n\geq 0}A_n$ is a standard graded $R$-subalgebra of $B$. Then there exists a constant $\alpha \in \mathbb{R}_{>0}$ such that $$l_R\left(H^0_{m_R}\left(\dfrac{B_n}{A_n}\right)\right)<\alpha n^{\dim B - 1}$$ for all $n\geq 1$.
 \end{lemma}
 \begin{proof}
We shall first show that $$l_R\left(H^0_{m_R}\left(\dfrac{B_n}{A_n}\right)\right) \leq  l_R\left(H^0_{m_R}(A_n)\right) + \sum\limits_{i=0}^{n-1}l_R\left(H^0_{m_R}\left(\dfrac{A_i B_{n-i}}{A_{i+1} B_{n-i-1}}\right)\right).$$ For every i, $0\leq i \leq n-2$, there exists a short exact sequence as follows $$0 \to \dfrac{A_{i+1}B_{n-i-1}}{A_n} \to \dfrac{A_{i}B_{n-i}}{A_n} \to \dfrac{A_i B_{n-i}}{A_{i+1}B_{n-i-1}} \to 0.$$ The function $l_R(H^0_{m_R}(-))$ is subadditive on short exact sequences, so we obtain that
\begin{align}
    l_R\left(H^0_{m_R}\left(\dfrac{B_n}{A_n}\right)\right) &\leq l_R\left(H^0_{m_R}\left(\dfrac{B_n}{A_1B_{n-1}}\right)\right) + l_R\left(H^0_{m_R}\left(\dfrac{A_1B_{n-1}}{A_n}\right)\right)\nonumber\\
    &\leq l_R\left(H^0_{m_R}\left(\dfrac{B_n}{A_1B_{n-1}}\right)\right) + l_R\left(H^0_{m_R}\left(\dfrac{A_1 B_{n-1}}{A_2B_{n-2}}\right)\right) + l_R\left(H^0_{m_R}\left(\dfrac{A_2 B_{n-2}}{A_n}\right)\right)\nonumber\\
    &\vdots\nonumber\\
    &\leq \sum\limits_{i=0}^{n-1}l_R\left(H^0_{m_R}\left(\dfrac{A_i B_{n-i}}{A_{i+1} B_{n-i-1}}\right)\right)\nonumber\\
    &\leq l_R\left(H^0_{m_R}(A_n)\right) + \sum\limits_{i=0}^{n-1}l_R\left(H^0_{m_R}\left(\dfrac{A_i B_{n-i}}{A_{i+1} B_{n-i-1}}\right)\right) \label{bound}.
\end{align}
Consider the associated graded ring of $B$ with respect to the ideal $I := A_1B$, i.e.
\begin{align*}
    \mathrm{gr}_I B &:= \bigoplus\limits_{n=0}^{\infty}\dfrac{I^n}{I^{n+1}} = \bigoplus\limits_{n=0}^{\infty} \dfrac{A_nB}{A_{n+1}B}.
\end{align*}
Here $\mathrm{gr}_I B$ is endowed with the "internal grading" as introduced in \cite[$2.3$ and $3.1$]{BJ}, i.e. $$\left(\mathrm{gr}_I B\right)_n =
\begin{cases}
R, &n=0\\
\left(\bigoplus\limits_{i=0}^{n-1} \dfrac{A_i B_{n-i}}{A_{i+1} B_{n-i-1}}\right)\oplus A_n, &n\geq 1
\end{cases}
$$
We now argue as in \cite{BJ}. With this internal grading, $\mathrm{gr}_I B$ becomes a standard graded $R$-algebra and $H^0_{m_R}\left(\mathrm{gr}_I B\right)$ is a finitely generated graded ideal of $\mathrm{gr}_I B$ which is annihilated a power of $m_R$. The sum on the right hand side of (\ref{bound}) gives rise to the Hilbert polynomial of $H^0_{m_R}\left(\mathrm{gr}_I B\right)$ and this polynomial has degree at most $$\dim \left(\mathrm{gr}_I B\right) - 1 \leq \dim B -1,$$ from which the conclusions of the lemma follows. 
\end{proof}
\begin{lemma}\label{dimless}
Suppose that $R$ is a universally catenary local ring with maximal ideal $m_R$, $B =\bigoplus\limits_{n\geq 0}B_n$ is a reduced standard graded $R$-algebra and $A = \bigoplus\limits_{n\geq 0}A_n$ is a standard graded $R$-subalgebra of $B$. Then $\dim A \leq \dim B$.
\end{lemma}
\begin{proof}
Let $P_1$,\ldots, $P_t$ be the minimal primes (which are necessarily homogeneous) of $B$. For every $1\leq i\leq t$, let $Q_i = P_i \cap A$. Since $B$ is reduced, we have $$\bigcap\limits_{i=1}^t P_i = 0 \implies \left(\bigcap\limits_{i=1}^t P_i\right)\cap A = 0 \implies \bigcap\limits_{i=1}^t Q_i = 0.$$ This shows that the minimal primes of $A$ appear amongst the primes $Q_1,\ldots,Q_t$. Now observe that for every $1\leq i\leq t$, there are graded inclusions $$\faktor{A}{Q_i} \subset \faktor{B}{P_i}$$ of standard graded $\faktor{R}{P_i \cap R}$-algebras, which are also domains. Let $m_{\faktor{A}{Q_i}}$ (respectively $m_{\faktor{B}{P_i}}$) denote the homogeneous maximal ideal of $\faktor{A}{Q_i}$ (respectively $\faktor{B}{P_i}$). As $\faktor{A}{Q_i}$ is universally catenary, we obtain from the dimension formula \cite{Mat1}[Theorem $23$, page $84$] that
\begin{align*}
    \mathrm{ht}\left(m_{\faktor{B}{P_i}}\right) &= \mathrm{ht}\left(m_{\faktor{A}{Q_i}}\right) + \mathrm{tr.deg.}_{QF\left(\faktor{A}{Q_i}\right)} QF\left(\faktor{B}{P_i}\right)\\
    \implies \dim \faktor{B}{P_i} &= \dim \faktor{A}{Q_i} + \mathrm{tr.deg.}_{QF\left(\faktor{A}{Q_i}\right)} QF\left(\faktor{B}{P_i}\right)\\
    \implies \dim \faktor{B}{P_i} &\geq \dim \faktor{A}{Q_i}.
\end{align*}
Using the definition of Krull dimension and the above inequality, we can conclude that $$\dim A = \max_{1\leq i\leq t} \left\{\dim \faktor{A}{Q_i}\right\} \leq \max_{1\leq i\leq t} \left\{\dim \faktor{B}{P_i}\right\} = \dim B.$$
\end{proof}
\begin{theorem}\label{maintheorem}
Suppose that $R$ is an excellent local ring with maximal ideal $m_R$, $B =\bigoplus\limits_{n\geq 0}B_n$ is a reduced standard graded $R$-algebra and $A = \bigoplus\limits_{n\geq 0}A_n$ is a standard graded $R$-subalgebra of $B$. Also assume that if $P$ is a minimal prime ideal of $B$ then $P\cap R \neq m_R$. Then $$\varepsilon\left(A\mid B\right) = \lim\limits_{n\to \infty}\left( \dfrac{(\dim B - 1)!}{n^{\dim B -1}}l_R\left(\faktor{(A_n\colon_{B_n}m_R^{\infty})}{A_n}\right)\right)$$ exists as a finite limit.
\end{theorem}
\begin{proof}
 There are graded inclusions of graded $R$-algebras as follows: $$A=\bigoplus\limits_{n=0}^{\infty}A_n \hookrightarrow A^{\prime}:=\bigoplus\limits_{n=0}^{\infty} (A_n\colon_{B_n}m_R^{\infty})\hookrightarrow B=\bigoplus\limits_{n=0}^{\infty}B_n.$$ Let $I=A_1B$, the ideal generated by $A_1$ in $B$. By \cite[Theorem $3.4.$]{IS}, for all $n\geq 1$, there exist irredundant primary decompositions $$I^n = \bigcap\limits_{i=1}^t q_i(n)$$ and a positive integer $c_0$ such that $$\left(\sqrt{q_i(n)}\right)^{c_0n}\subset q_i(n)$$ for all $n\geq 1$ and $1\leq i\leq t$. Suppose that $c\geq c_0$. Since
\begin{align*}
    I^n\colon_B\, (m_RB)^{\infty} &= \left(\bigcap\limits_{i=1}^t q_i(n)\right)\colon_B\, (m_RB)^{\infty}\\
    &= \bigcap\limits_{i=1}^t\left(q_i(n)\colon_B\,(m_RB)^{\infty}\right)\\
    &= \left(\bigcap_{\substack{1\leq i\leq t\\ m_RB\not\subset \sqrt{q_i(n)}}}\left(q_i(n)\colon_B\,(m_RB)^{\infty}\right)\right)\cap \left(\bigcap_{\substack{1\leq i\leq t\\ m_RB\subset \sqrt{q_i(n)}}}\left(q_i(n)\colon_B\,(m_RB)^{\infty}\right)\right)\\
    &= \bigcap_{\substack{1\leq i\leq t\\ m_RB\not\subset \sqrt{q_i(n)}}}\left(q_i(n)\colon_B\,(m_RB)^{\infty}\right)\\
    &= \bigcap_{\substack{1\leq i\leq t\\ m_RB\not\subset \sqrt{q_i(n)}}}q_i(n)
\end{align*}
we have that 
\begin{equation}\label{lc1}
    (m_RB)^{cn}\cap \left(I^n\colon_B \,(m_RB)^{\infty}\right)
    \subset \left(\bigcap_{\substack{1\leq i\leq t\\ m_RB\not\subset \sqrt{q_i(n)}}}q_i(n)\right)\cap \left(\bigcap_{\substack{1\leq i\leq t\\ m_RB\subset \sqrt{q_i(n)}}}q_i(n)\right) = I^n
\end{equation}
for all positive integers $n$. Also
\begin{align}
    I^n\cap B_n &= \left(\bigoplus\limits_{k=0}^{\infty}A_nB_k\right)\cap B_n = A_n,\label{lc2}\\
    (m_RB)^{cn}\cap B_n &= \left(\bigoplus\limits_{k=0}^{\infty}m_R^{cn}B_k\right)\cap B_n = m_R^{cn}B_n.\label{lc3}
\end{align}
We further observe that
\begin{align}
    \left(I^n\colon_B \,(m_RB)^{\infty}\right)\cap B_n
    &= \left(\bigoplus\limits_{k=0}^{\infty}\left(A_nB_k \colon_{B_{k+n}}\, m_R^{\infty}\right)\right)\cap B_n \nonumber\\
    &= (A_n\colon_{B_n}m_R^{\infty})\label{lc4}.
\end{align}
Thus for all $n\geq 1$, and $c\geq c_0$
\begin{align*}
    m_R^{cn}B_n\cap A_n &\subset m_R^{cn}B_n\cap \left(A_n\colon_{B_n}m_R^{\infty}\right)\\
    &= m_R^{cn}B_n\cap \left[(m_RB)^{cn}\cap \left(I^n\colon_B \,(m_RB)^{\infty}\right)\right]\cap B_n\quad(\text{using $(\ref{lc3})$ and $(\ref{lc4})$})\\
    &\subset m_R^{cn}B_n\cap\left(I^n\cap B_n\right)\quad(\text{using $(\ref{lc1})$})\\
    &= m_R^{cn}B_n\cap A_n \quad(\text{using $(\ref{lc2})$}).
\end{align*}
Hence we conclude that
\begin{equation}\label{lc5}
    (m_R^{cn}B_n)\cap A_n = m_R^{cn}B_n\cap \left(A_n\colon_{B_n}m_R^{\infty}\right)
\end{equation}
for all $c\geq c_0$ and $n\geq 1$. Using $(\ref{lc5})$, for all $c\geq c_0$ and $n\geq 1$, there are short exact sequences of $R$-modules
\begin{equation}\label{lc6}
    0\to \faktor{A_n}{(m_R^{cn}B_n)\cap A_n}\to \faktor{\left(A_n\colon_{B_n}m_R^{\infty}\right)}{(m_R^{cn}B_n)\cap \left(A_n\colon_{B_n}m_R^{\infty}\right)}\to\faktor{\left(A_n\colon_{B_n}m_R^{\infty}\right)}{A_n}\to 0.
\end{equation}
From Lemma \ref{bound2}, we know that there exists a positive constant $\alpha$ such that
\begin{equation}\label{lc7}
    l_R\left(\faktor{\left(A_n\colon_{B_n}m_R^{\infty}\right)}{A_n}\right)<\alpha n^{\dim B -1}
\end{equation}
for all $n\geq 1$. Combining Lemma \ref{bound1} and Lemma \ref{dimless} give us that for a fixed $c$, there exists a positive constant $\gamma_c$ such that
\begin{equation}\label{lc8}
    l_R\left(\faktor{A_n}{(m_R^{cn}B)\cap A_n}\right)<\gamma_c n^{\dim A -1} \leq \gamma_c n^{\dim B -1} 
\end{equation}
for all $n\geq 1$. Using the short exact sequence $(\ref{lc6})$ and bounds $(\ref{lc7})$ and $(\ref{lc8})$, we obtain that
\begin{align*}
    l_R\left(\faktor{\left(A_n\colon_{B_n}m_R^{\infty}\right)}{(m_R^{cn}B_n)\cap \left(A_n\colon_{B_n}m_R^{\infty}\right)}\right) &= l_R\left(\faktor{A_n}{(m_R^{cn}B_n)\cap A_n}\right) + l_R\left(\faktor{\left(A_n\colon_{B_n}m_R^{\infty}\right)}{A_n}\right)\nonumber\\
    &< (\alpha + \gamma_c)n^{\dim B - 1}
\end{align*}
for all $c\geq c_0$ and for all positive integers $n$. Let $P$ be a minimal prime ideal of $B$. If $$\faktor{A_1}{P\cap A_1} = 0$$ then using $A_n = A_1^n$, we get that $$\faktor{A_n}{P\cap A_n} = 0$$ for all $n\geq 1$. Similarly if $$\faktor{\left(A_1\colon_{B_1}m_R^{\infty}\right)}{P\cap\left(A_1\colon_{B_1}m_R^{\infty}\right)} = 0$$ then again using $A_n = A_1^n$, we get that $$\faktor{\left(A_n\colon_{B_n}m_R^{\infty}\right)}{P\cap\left(A_n\colon_{B_n}m_R^{\infty}\right)} = 0$$ for all $n\geq 1$. Now from Theorem \ref{main2}, we can conclude that the limits
$$\lim\limits_{n\to \infty}\dfrac{l_R\left(\faktor{A_n}{(m_R^{cn}B_n)\cap A_n}\right)}{n^{\dim B-1}}\;\;\text{and}\;\;\lim\limits_{n\to \infty}\dfrac{l_R\left(\faktor{\left(A_n\colon_{B_n}m_R^{\infty}\right)}{(m_R^{cn}B_n)\cap \left(A_n\colon_{B_n}m_R^{\infty}\right)}\right)}{n^{\dim B - 1}}$$ exist. Finally from the short exact sequence $(\ref{lc6})$, we get that $$\dfrac{l_R\left(\faktor{\left(A_n\colon_{B_n}m_R^{\infty}\right)}{A_n}\right)}{n^{\dim B - 1}} = \dfrac{l_R\left(\faktor{\left(A_n\colon_{B_n}m_R^{\infty}\right)}{(m_R^{cn}B_n)\cap \left(A_n\colon_{B_n}m_R^{\infty}\right)}\right)}{n^{\dim B - 1}} - \dfrac{l_R\left(\faktor{A_n}{(m_R^{cn}B_n)\cap A_n}\right)}{n^{\dim B - 1}}$$ and by taking $n\to \infty$, we obtain the conclusions of the theorem.
\end{proof}
\begin{proposition}\label{Hilb}
Let $B = \oplus_{n\geq 0}B_n$ be a reduced standard graded algebra over a field $k$ and let $A = \oplus_{n\geq 0}$ be a standard graded $k$-subalgebra of $B$. Then $$\varepsilon(A\mid B) = \lim_{n\to \infty}\left(\dfrac{(\dim B-1)!}{n^{\dim B-1}}l_R\left(\dfrac{B_n}{A_n}\right)\right)$$ exists as a finite limit.
\end{proposition}
\begin{proof}
The maximal ideal $m$ of $k$ is zero since $k$ is a field. For all $n\geq 0$, we have isomorphisms $$H^0_m\left(\dfrac{B_n}{A_n}\right) \cong \dfrac{B_n}{A_n}.$$ From Lemma \ref{dimless}, we know that $\dim A \leq \dim B$. Therefore the function $$n \mapsto l_R\left(\dfrac{B_n}{A_n}\right) = l_R(B_n) - l_R(A_n)$$ is eventually a polynomial in $n$ of degree at most $\dim B - 1$, from which the proposition follows.
\end{proof}
\section{Applications}
Corollary $\ref{cor1}$ is an important special case of Theorem $3.2.$ in \cite{DC1}.
\begin{corollary}\label{cor1}
Let $R$ be an excellent reduced local ring of dimension $d>0$ with maximal ideal $m_R$ and $E$ is a submodule of a finite free $R$-module $F=R^n$. Let $B = R[F]$ be the symmetric algebra of $F$ over $R$, which is isomorphic to the standard graded polynomial ring $$B = R[x_1,\ldots,x_n] = \bigoplus\limits_{k=0}^{\infty} F^k$$ over $R$. We may identify $E$ with a submodule $E^1$ of $B_1$ and let $$A = R[E] = \bigoplus\limits_{k=0}^{\infty}E^k$$ be the graded $R$-subalgebra of $B$ generated by $E^1$ over $R$. Then $$\varepsilon\left(A\mid B\right) = \lim\limits_{k\to\infty}\dfrac{(d+n-1)!}{k^{d+n-1}}l_R\left(\faktor{(E^k\colon_{F^k}m_R^{\infty})}{E^k}\right)$$ exists as a finite limit.
\end{corollary}
\begin{proof}
Let $P_1,\ldots,P_t$ be the minimal primes of $R$. As $B$ is a polynomial ring over $R$, its minimal primes are $P_1B,\cdots,P_tB$. Also $\dim R>0$, so that $$P_iB \cap R = P_i \neq m_R$$ for all $i$. Moreover $\dim B = d+n$ and the corollary now follows from Theorem $\ref{maintheorem}$.
\end{proof}
Corollary $\ref{cor2}$ is an important special case of Corollary $6.3.$ in \cite{DC6}.
\begin{corollary}\label{cor2}
Suppose that $R$ is an excellent reduced local ring of dimension $d>0$ with maximal ideal $m_R$ and $I$ is an ideal in $R$. Then $$\varepsilon(I)= \lim\limits_{n\to\infty} \dfrac{d!}{n^d}l_R\left(\faktor{\left(I^n\colon_{R}\,m_R^{\infty}\right)}{I^n}\right) $$ exists as a finite limit.
\end{corollary}
\begin{proof}
 We let $$A:= R[xI] = \bigoplus\limits_{n\geq 0}x^nI^n \subset B:= R[x] = \bigoplus\limits_{n\geq 0}x^nR$$ where $x$ is an indeterminate. There are isomorphisms $$\dfrac{\left(I^n\colon_R\,m_R^{\infty}\right)}{I^n} \cong H^0_{m_R}\left(\dfrac{R}{I^n}\right) \cong H^0_{m_R}\left(\dfrac{B_n}{A_n}\right)$$ where $B_n = x^n R$ and $A_n = x^nI^n$. Any minimal prime ideal of $R[x]$ is of the form $PR[x]$ where $P$ is a minimal prime ideal of $R$. Let $P$ be a minimal prime ideal of $R$. Then $$PR[x]\cap R = P \neq m_R$$ as $\dim R > 0$. Moreover $\dim R[x] = d+1$ and the conclusions of the corollary now follow from Theorem \ref{maintheorem}.
\end{proof}

\bibliographystyle{plain}
\bibliography{main}
\end{document}